\documentclass[12pt,a4paper]{amsart}
\usepackage{amsmath,amsthm,amssymb}

\usepackage[shortlabels]{enumitem}

\usepackage[totalwidth=15.75cm,totalheight=22.275cm]{geometry}

\usepackage{graphicx}

\usepackage{hyperref}
\usepackage[usenames, dvipsnames]{color}
\usepackage[flushmargin]{footmisc}

\usepackage{gensymb}

\newcommand{\der}{{\rm d}}

\numberwithin{equation}{section}

\makeatletter
\def\swappedhead#1#2#3{%
  \thmnumber{\@upn{\the\thm@headfont #2\@ifnotempty{#1}{.~}}}%
  \thmname{#1}%
  \thmnote{ {\the\thm@notefont(#3)}}}
\makeatother
\swapnumbers

\theoremstyle{plain}
\newtheorem{thm}{Theorem}[section]%
\newtheorem{lem}[thm]{Lemma}%
\newtheorem{cor}[thm]{Corollary}%
\theoremstyle{definition}
\newtheoremstyle{claimstyle}%
   {}%             space above
   {}%             space below
   {\normalfont}%     body font
   {}%                indent
   {\itshape}%        header font
   {.}%               punctuation
   { }%               space after head
   {\thmnote{#3}}%    typeset note only.
   
\theoremstyle{claimstyle}
\newtheorem*{varclaim}{}

\newcommand*{\defeq}{\mathrel{\vcenter{\baselineskip0.5ex \lineskiplimit0pt
                     \hbox{\scriptsize.}\hbox{\scriptsize.}}}%
                     =}

\renewcommand{\theta}{\vartheta}
\renewcommand{\phi}{\varphi}

\newcommand{\C}{{\mathbb{C}}}

\newcommand{\R}{{\mathbb{R}}}

\newcommand{\im}{\operatorname{Im}}
\newcommand{\re}{\operatorname{Re}}

\newcommand{\HH}{\mathbb{H}}
\newcommand{\B}{\mathcal{B}}

\makeatletter
\@namedef{subjclassname@2020}{%
  \textup{2020} Mathematics Subject Classification}
\makeatother

\title{The Eremenko-Lyubich constant}

\begin{document}

\author{Lasse Rempe} 
\address{Dept. of Mathematical Sciences \\
	 University of Liverpool \\
   Liverpool L69 7ZL\\
   UK \\ 
	 \textsc{\newline \indent 
	   \href{https://orcid.org/0000-0001-8032-8580%
	     }{\includegraphics[width=1em,height=1em]{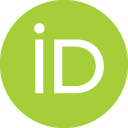} {\normalfont https://orcid.org/0000-0001-8032-8580}}%%
	       }}
%	ORCiD: 0000-0001-8032-8580}
\date{\today}
\email{l.rempe@liverpool.ac.uk}
\subjclass[2020]{Primary 30D15, Secondary 30D05, 37F10.}

\begin{abstract}Eremenko and Lyubich proved that an
  entire function $f$ whose set of singular values is bounded 
  is expanding at points where $\lvert f(z)\rvert$ is large. These expansion
  properties have been at the centre of the subsequent study of this class of
  functions, now called the \emph{Eremenko--Lyubich class}. We improve
  the estimate of Eremenko and Lyubich, and show that the new estimate is
  asymptotically optimal. As a corollary, we obtain an elementary proof that 
  functions in the Eremenko--Lyubich class have lower order at least $1/2$. 
  \end{abstract}

\maketitle

\section{Introduction}
 Eremenko and Lyubich~\cite{alexmisha} introduced the class
   $\B$ of transcendental entire functions whose set of critical and asymptotic
   values is bounded. The dynamics of functions in this
   class, now called the \emph{Eremenko--Lyubich class},
   has been studied intensively in recent years; see e.g.\ 
   \cite{boettcher,strahlen,sixsmithELnew,bishopfolding,bishopELmodels,martipeteshishikura}
   and the survey~\cite{sixsmithELsurvey}. The key property of functions in this 
   class is that they satisfy strong expansion properties near infinity.
   More precisely, let $R>0$ be so large that 
   $\lvert f(0)\rvert \leq R$ and such that all critical and asymptotic values of $f$
   have modulus at most $R$. Then there is a holomorphic function
     \[ F \colon \mathcal{T} \to \HH_{\ln R} \defeq \{ x + iy\colon x > \ln R \} \]
     with $\exp\circ F = f\circ \exp$, where 
       \[ \mathcal{T} = \{\zeta \in\C\colon \lvert f(\exp(\zeta) )\rvert > R \} =
                      \{\zeta\in\C\colon \re F(\zeta) > \ln R\}. \] 
   The map $F$ is a conformal isomorphism when restricted to any connected component
     of $\mathcal{T}$; these components are also called the \emph{(logarithmic) tracts}
      of $F$. 
      
    Eremenko and Lyubich~\cite[Lemma~1]{alexmisha}
     use the Koebe quarter theorem to show that 
    \begin{equation}\label{eqn:EL}
       \lvert F'(\zeta) \rvert \geq \frac{\re F(\zeta) - \ln R}{4\pi}
    \end{equation}
    for all $\zeta\in \mathcal{T}$; in other words, 
         \begin{equation}\label{eqn:ELexp}
        \lvert f'(z)\rvert \geq \frac{(\ln \lvert f(z)\rvert - \ln R)\cdot \lvert f(z)\rvert }{4\pi \lvert z\rvert}
     \end{equation}
      whenever $\lvert f(z)\rvert > R$.

  The inequality~\eqref{eqn:EL} has played a fundamental role in the
   study of the class $\B$. It therefore seems natural
    to ask whether the above estimates can be improved. 
    The purpose of this
    note is to establish the optimal constant in~\eqref{eqn:EL} and~\eqref{eqn:ELexp}. 
  \begin{thm}[Optimal expansion estimates]\label{thm:newEL}
    Let $f$ and $F$ be as above. Then 
           \begin{equation}\label{eqn:newEL} \lvert F'(\zeta )\rvert \geq \frac{\re F(\zeta) - \ln R}{2} 
           \end{equation}
   for all $\zeta\in\mathcal{T}$, and hence 
         \begin{equation}\label{eqn:newELexp}
        \lvert f'(z)\rvert \geq \frac{(\ln \lvert f(z)\rvert - \ln R)\cdot \lvert f(z)\rvert }{2 \lvert z\rvert}
     \end{equation}
      whenever $\lvert f(z)\rvert > R$.
      
   Moreover, the constant $2$ is optimal. Indeed, for 
   $f(z) = \cos \sqrt{z}$, 
    \begin{equation}\label{eqn:cosoptimal} \lim_{r\to +\infty} \frac{\lvert f'(-r)\rvert\cdot r}{ (\ln \lvert f(-r)\rvert) \cdot \lvert f(-r)\rvert}
         =  \frac{1}{2}. \end{equation}
  \end{thm} 
  
   Every component of $\mathcal{V}\defeq \exp(\mathcal{T})$ is simply-connected, and $f$ is bounded on $\partial\mathcal{V}$. By an old result
     of Wiman \cite{wiman1903}, it follows that 
     $f$ must have order $\rho(f) \geq 1/2$, where 
        \[ \rho(f) = \limsup_{r\to\infty} \frac{\ln \ln M(r,f)}{\ln r}. \]
      Here $M(r,f) = \max_{\lvert z\rvert = r} \lvert f(z)\rvert$.
        That $\rho(f)\geq 1/2$ also 
         follows 
         from the famous $\cos \pi \rho$ theorem, proved later by Wiman \cite{wimancospirho}. See also \cite[p.~119]{haymanmeromorphic},
           and compare~\cite[p.\ 1788]{langleymultiplepoints} and~\cite[Proof~of~Corollary~2]{bergweilereremenkosingularities}.
        In fact, $f$ must even have \emph{lower order} at least $1/2$, i.e.
        \begin{equation}\label{eqn:lowerorder} \liminf_{r\to\infty} \frac{\ln \ln M(r,f)}{\ln r} \geq \frac{1}{2}. \end{equation}
     According to Heins~\cite{heinsboundedminimummodulus}, this result is also due to Wiman; see
      \cite[Lemma~3.5]{ripponstallarddimensionmero}, where~\eqref{eqn:lowerorder} is proved using a version
      of the Ahlfors distortion theorem. 
      As an application of Theorem~\ref{thm:newEL}, we obtain a new elementary proof of~\eqref{eqn:lowerorder}. 
     
   \begin{cor}[Lower order of functions in $\B$]\label{cor:lowerorder}
     Let $f\in\B$. Then $f$ has lower order at least $1/2$.      
       More precisely, there are constants $c>0$ and $r_0>1$ such that 
         \[ \ln M(r,f) \geq c\cdot \sqrt{r}\]
      for $r\geq r_0$. 
   \end{cor} 

\subsection*{Acknowledgements} I am grateful to Walter Bergweiler and Alex Eremenko for helpful comments on a draft of this
 paper, and in particular for discussing  
 the history of 
 Corollary~\ref{cor:lowerorder}. I also thank Chris Bishop, Tania Gricel Ben\'itez L\'opez, Weiwei Cui, Dave Sixsmith and James Waterman for interesting
 discussions and comments concerning the paper. I am particularly grateful to David Mart\'i-Pete for a careful reading of an earlier manuscript; his 
 corrections and thoughtful comments improved the presentation of the article. 
      
\section{Proof of the estimate}

 \begin{lem}[Expansion in logarithmic coordinates]\label{lem:newEL}
   Suppose that $T\subset\C$ is a simply-connected domain such that
     $\exp|_T$ is injective, let $\rho \in\R$ and suppose that 
        $\phi \colon T\to\HH_{\rho}$
        is a conformal isomorphism. Then
           \[ \lvert \phi '(\zeta)\rvert \geq \frac{\re \phi(\zeta) - \rho}{2} 
           \]
      for all $\zeta\in T$. 
 \end{lem} 
 \begin{proof}
   Post-composing with a translation, we may assume that $\rho=0$. Let $\mathbb{D}$ denote
    the unit disc, and consider the M\"obius transformation  
       \[ M\colon\mathbb{D}\to\HH_0; \quad z\mapsto (1-z)/(1+z). \] 
         Let $\zeta\in T$. The map 
        \[ g\colon \mathbb{D} \to \exp(T-\zeta); \qquad z \mapsto \exp\bigl( \phi^{-1}\bigl( \re (\phi(\zeta))\cdot M(z) + i\cdot \im \phi(\zeta)\bigr) - \zeta \bigr) \]
   is a conformal isomorphism with $g(0)=1$. Since $0\notin \exp(T-\zeta)$, the Koebe quarter theorem implies that 
    \[ 4\geq \lvert g'(0) \rvert = \frac{\re(\phi(\zeta))\cdot \lvert M'(0)\rvert}{\lvert \phi'(\zeta)\rvert}
          = \frac{2\re (\phi(\zeta))}{\lvert \phi'(\zeta)\rvert}. \] 
  So $\lvert \phi'(\zeta)\rvert \geq \re \phi(\zeta)/2$, as claimed. 
\end{proof}

 \begin{proof}[Proof of Theorem~\ref{thm:newEL}]
   If $T$ is a connected component of $\mathcal{T}$, then $T$ and 
     $\phi \defeq F|_T$ satisfy the hypotheses of Lemma~\ref{lem:newEL}.
     This proves~\eqref{eqn:newEL}. To deduce~\eqref{eqn:newELexp} from~\eqref{eqn:newEL}, differentiate the functional equation
     $\exp \circ F = f\circ \exp$. 

    For $f(z) = \cos \sqrt{z}$ and $r>0$, we have 
     \[ \lvert f'(-r) \rvert = \frac{\lvert \sin i \sqrt{r}\rvert}{2\sqrt{r}} . \]
  As $r\to\infty$, 
     \[ \lvert \sin i \sqrt{r}\rvert = \frac{e^{\sqrt{r}}}{2} + O(e^{-\sqrt{r}}) = \lvert f(-r) \rvert + O(e^{-\sqrt{r}}). \] 
   Moreover, 
      \[ \ln \lvert f(-r)\rvert = \sqrt{r} - \ln 2 + O(1/\lvert f(x)\rvert^2). \]
      Hence
        \[\lvert f'(-r)\rvert - \frac{(\ln \lvert f(-r)\rvert + \ln 2) \cdot \lvert f(-r)\rvert}{2r} 
            = O(e^{-\sqrt{r}}/\sqrt{r}) \]
   as $r \to\infty$. We may rewrite this as 
   \[ 
   \frac{\lvert f'(-r)\rvert\cdot r}{ (\ln \lvert f(-r)\rvert) \cdot \lvert f(-r)\rvert} - \frac{1}{2}         
         =  \frac{\ln 2}{\ln \lvert f(-r)\rvert} + O\left(\frac{e^{-\sqrt{r}}\cdot r}{(\ln \lvert f(-r)\rvert)\cdot \lvert f(-r)\rvert \cdot \sqrt{r}}\right) = 
             O(e^{-\sqrt{r}}), \]   
   which proves~\eqref{eqn:cosoptimal}.
 \end{proof}

 For future reference, we also record the following reformulation of Lemma~\ref{lem:newEL} in terms of the hyperbolic density of $T$;
    see e.g.~\cite{beardonminda} for background on hyperbolic geometry.

 \begin{cor}[Hyperbolic density of logarithmic tracts]\label{cor:Testimate}
   Suppose that $T\subset\C$ is a simply-connected domain such that
     $\exp|_T$ is injective, and let $\rho_T$ denote the density of the hyperbolic
     metric of $T$. Then 
     \begin{equation}\label{eqn:logarithmicdensity} \rho_T(\zeta ) \geq \frac{1}{2} \quad \text{for all $\zeta\in T$.}\end{equation}
 \end{cor}
 \begin{proof}
   Let $\phi\colon T\to \HH_0$ be a conformal isomorphism. The hyperbolic density of $\HH_0$ is given by 
     $\rho_{\HH_0}(\omega ) = 1/\re \omega$. 
    The hyperbolic density $\rho_T$ satisfies
       \[ \rho_T(\zeta) =  \rho_{\HH_0}(\omega) \cdot \lvert \phi'(\zeta)\rvert = \frac{\lvert \phi'(\zeta) \rvert}{\re \phi(\zeta)} \geq \frac{1}{2} \]
       by Lemma~\ref{lem:newEL}. 
 \end{proof}

  The following implies Corollary~\ref{cor:lowerorder}.
    
  \begin{cor}[Lower order in logarithmic coordinates]
   Suppose that $T\subset\C$ is a simply-connected domain such that
     $\exp|_T$ is injective, let $\rho \in\R$ and suppose that 
        \[ \phi\colon T\to\HH_{\rho} \]
        is a conformal isomorphism with $\phi^{-1}(t)\to\infty$ as $t\to+\infty$, 
         $t > 0$. 
        Then there are constants $\gamma>0$ and $\sigma_0>0$ such that 
           \[
               \sup_{\re \zeta = \sigma} \ln \re \phi(\zeta) \geq  \frac{\sigma}{2} - \gamma
           \]
      for $\sigma \geq \sigma_0$. 
  \end{cor}    
  \begin{proof}
    Without loss of generality, we may suppose that $\rho = 0$. 
     Set $\zeta_0\defeq \phi^{-1}(1)$ and $\gamma\defeq \re \zeta_0/2$ and $\sigma_0 \defeq \re \zeta_0$.
      Let $\sigma>\sigma_0$; by the assumption on $\phi$ and the intermediate value theorem, 
      there is 
     $t>1$ such that $\re \phi^{-1}(t) = \sigma$. Set $\zeta \defeq \phi^{-1}(t)$. By Lemma~\ref{lem:newEL}, 
     \[ \lvert \zeta - \zeta_0 \rvert \leq \int_{1}^t \lvert (\phi^{-1})'(x) \rvert \der x 
             \leq \int_1^t \frac{2 \der x}{x} = 2 \ln t = 2 \ln \re \phi(\zeta). \]
     
   So 
      \[ \ln \re \phi(\zeta) \geq \frac{\lvert \zeta - \zeta_0\rvert}{2} = \frac{\sigma}{2} - \gamma. \qedhere \] 
  \end{proof}

 \end{document}